\documentclass[english]{article}
\usepackage[T1]{fontenc}
\usepackage[latin9]{inputenc}
\usepackage{mathrsfs}
\usepackage{amsmath}
\usepackage{amsthm}
\usepackage{amssymb}
\usepackage{stmaryrd}
\usepackage{setspace}
\makeatletter
\theoremstyle{plain}
\newtheorem{thm}{\protect\theoremname}[section]
  \theoremstyle{remark}
  \newtheorem*{notation*}{\protect\notationname}
  \theoremstyle{plain}
  \newtheorem{lem}[thm]{\protect\lemmaname}
  \theoremstyle{plain}
  \newtheorem{prop}[thm]{\protect\propositionname}
  \theoremstyle{remark}
  \newtheorem*{acknowledgement*}{\protect\acknowledgementname}
\newcommand{\lyxaddress}[1]{
\par {\raggedright #1
\vspace{1.4em}
\noindent\par}
}

\usepackage{enumitem}
\setenumerate{label=(\roman{enumi})}

\makeatother

\usepackage{babel}
  \providecommand{\acknowledgementname}{Acknowledgement}
  \providecommand{\lemmaname}{Lemma}
  \providecommand{\notationname}{Notation}
  \providecommand{\propositionname}{Proposition}
\providecommand{\theoremname}{Theorem}

\begin{document}

\title{On the failure of Ornstein theory in the finitary category}

\author{Uri Gabor\thanks{Supported by ISF grant 1702/17}}
\maketitle
\begin{abstract}
We show the invalidity of finitary counterparts for three classification
theorems: The preservation of being a Bernoulli shift through factors,
Sinai's factor theorem, and the weak Pinsker property. We construct
a finitary factor of an i.i.d. process which is not finitarily isomorphic
to an i.i.d. process, showing that being finitarily Bernoulli is not
preserved through finitary factors. This refutes a conjecture of M.
Smorodinsky \cite{key-8}, which was first suggested by D. Rudolph
\cite{key-7}. We further show that any ergodic system is isomorphic
to a process none of whose finitary factors are i.i.d. processes,
and in particular, there is no general finitary Sinai's factor theorem
for ergodic processes. An immediate consequence of this result is
the invalidity of a finitary weak Pinsker property, answering a question
of G. Pete and T. Austin \cite{key-1}.
\end{abstract}

\section{Introduction}

A (measure preserving) system $\mathcal{X}=(\mathbf{X},\mathscr{B},\mu,T)$
is a probability space $(\mathbf{X},\mathscr{B},\mu)$ endowed with
a measure preserving transformation $T$. Given two systems $\mathcal{X}=(\mathbf{X},\mathscr{B},\mu,T_{\mathbf{X}})$
and $\mathcal{Y}=(\mathbf{Y},\mathscr{C},\nu,T_{\mathbf{Y}})$, a
measurable function $\phi:\mathbf{X}\rightarrow\mathbf{Y}$ is said
to be a factor map, and $\mathcal{Y}$ is a factor of $\mathcal{X}$,
if $\phi$ is translation equivariant (i.e. $\phi\circ T_{\mathbf{Y}}=T_{\mathbf{X}}\circ\phi$)
and $\phi_{*}\nu=\mu$. If furthermore, $\phi$ is invertible, then
$\phi$ is said to be an isomorphism, and $\mathcal{X}$ and $\mathcal{Y}$
are said to be isomorphic.

A process is a system $\mathcal{X}$ where $\mathbf{X}=A^{\mathbb{Z}}$
for some set $A$ and $T_{\mathbf{X}}$ is the left shift operator.
In this case, we abbreviate $\mathcal{X}=(A^{\mathbb{Z}},\mu)$, and
$S$ stands for the left shift operator. Given two processes $\mathcal{X}=(A^{\mathbb{Z}},\mu)$
and $\mathcal{Y}=(B^{\mathbb{Z}},\nu)$, a factor map $\phi:A^{\mathbb{Z}}\rightarrow B^{\mathbb{Z}}$
is said to be finitary, and $\mathcal{Y}$ is a finitary factor of
$\mathcal{X}$, if for $\mu$-a.e. $x\in A^{\mathbb{Z}}$ there is
some $r=r(x)$ s.t. the projection of $\phi$ on its zero coordinate,
$\phi_{0}$, is constant up to $\mu$-measure zero on $[x_{-r}^{r}]=\left\{ x'\in A^{\mathbb{Z}}:\,\forall-r\leq i\leq r,\,x'_{i}=x_{i}\right\} $.
If in addition $\phi$ is invertible with a finitary inverse, then
$\phi$ is a finitary isomorphism and $\mathcal{X}$ and $\mathcal{Y}$
are said to be finitarily isomorphic. Systems of particular interest
in this paper are i.i.d. processes $(B^{\mathbb{Z}},p^{\times\mathbb{Z}})$,
where $p=(p_{b})_{b\in B}$ is a probability measure on $B$ and $p^{\times\mathbb{Z}}$
is the product measure, and Bernoulli systems, which are systems isomorphic
to some i.i.d. process.

Several years after D.S. Ornstein proved the Isomorphism theorem of
Bernoulli systems, saying that two Bernoulli systems of equal entropy
are isomorphic, M. Keane and M. Smorodinsky proved a finitary version
of it, showing that two i.i.d. processes of equal entropy are finitarily
isomorphic. Besides the Isomorphism theorem, one may wonder whether
other theorems in the area of measure theoretic classification theory
might have finitary versions. In this paper we deal with the invalidity
of finitary counterpart for three of these theorems: The Sinai's factor
Theorem \cite{key-3,key-6}, The Weak Pinsker property \cite{key-1},
and that factors of Bernoulli systems are Bernoulli \cite{key-2}.
Among those three theorems, the last one is probably the oldest candidate
for having a finitary counterpart: In \cite{key-7}, Rudolph asked
whether any finitary factor of an i.i.d. process is finitarily isomorphic
to an i.i.d. process. A decade later, M. Smorodinsky proved in \cite{key-8}
that any continuous factor of an i.i.d. process (and more generally,
any finitely dependent process) is finitarily isomorphic to some i.i.d.
process. This result falls short of the finitary case, since finitary
maps are just maps that are continuous on a some full measure set.
Indeed, it was conjectured there that the result is valid for finitary
factors too (see also the work of Shea \cite{key-12} and Lazowski
and Shea \cite{key-11} on this conjecture). Our first theorem refutes
this conjecture:
\begin{thm}
\label{thm:Refute Conj}There exists a finitary factor of an i.i.d.
process which is not finitarily isomorphic to an i.i.d. process.
\end{thm}

What we actually prove is the following theorem, for which Theorem
\ref{thm:Refute Conj} is an immediate corollary of it.
\begin{thm}
\label{thm:There-exists-a}There exists a process of entropy greater
than $\log2$ which is a finitary factor of an i.i.d. process, but
none of whose finitary factors of entropy greater than $\log2$ is
an i.i.d. process. In particular, this process is not finitarily isomorphic
to an i.i.d. process.
\end{thm}

We further investigate how common it is for a process to finitarily
factor onto an i.i.d. process. Sinai's factor theorem \cite{key-3,key-6}
asserts that any ergodic system factors onto any Bernoulli system
of lesser or equal entropy. A finitary version of this theorem would
ask for a finitary factor map from any ergodic process onto any i.i.d.
process of lesser or equal entropy. As the following theorem shows,
any ergodic system of finite entropy has an isomorphic process for
which this finitary version is invalid:
\begin{thm}
\label{thm:Finitary Sinai}Any ergodic system of positive entropy
is isomorphic to a process none of whose finitary factors is an i.i.d.
process.
\end{thm}

Closely related to Sinai's factor theorem is the weak Pinsker property
- a long standing conjecture proved recently by T. Austin \cite{key-1},
which asserts that for any ergodic system and a Bernoulli system of
strictly smaller entropy, the former system splits as a product of
the latter with another system. Here, a finitary version of it would
ask whether every ergodic process is finitarily isomorphic to the
product of an i.i.d. process with a suitable low-entropy process.
The question whether this finitary weak Pinsker property takes place
in general was suggested by G. Pete and T. Austin \cite{key-1}. A
negative answer to their question is an immediate corollary of Theorem
\ref{thm:Finitary Sinai}: For if we would assume there is for any
process $\mathbf{X}$ a finitary isomorphism $\phi:\mathbf{X}\rightarrow\mathbf{B}\times\mathbf{Y}$
with $\mathbf{B}$ being an i.i.d. process, then composing the projection
$\pi_{1}:\mathbf{B}\times\mathbf{Y}\rightarrow\mathbf{B}$ together
with $\phi$, gives a finitary factor map $\pi_{1}\circ\phi:\mathbf{X}\rightarrow\mathbf{B}$,
contradicting Theorem \ref{thm:Finitary Sinai}.
\begin{notation*}
For $x\in A^{\mathbb{Z}}$ and integers $m<n$, write
\begin{align*}
x_{m}^{n} & =(x_{m},...,x_{n})\\
\left[x_{m}^{n}\right] & =\left\{ y\in A^{\mathbb{Z}}:\,y_{m}^{n}=x_{m}^{n}\right\} 
\end{align*}
Given a finitary map $\phi:\left(A^{\mathbb{Z}},\mu\right)\rightarrow\left(B^{\mathbb{Z}},\nu\right)$,
denote by $\pi_{0}:B^{\mathbb{Z}}\rightarrow B$ the projection on
the zero coordinate, and let
\[
\phi_{0}:=\phi\circ\pi
\]
\[
R_{n}(x):=\min\left\{ r\geq0:\phi_{0}\text{ is \ensuremath{\mu}-a.e. constant on }\left[\left(S^{n}x\right)_{-r}^{r}\right]\right\} 
\]
and 
\[
\left[\phi\left(x_{m}^{n}\right)\right]:=\left\{ y\in B^{\mathbb{Z}}:\begin{gathered}\forall i\:\text{s.t.}\,[i-R_{i}(x),i+R_{i}(x)]\subset[m,n],\\
y_{i}=\phi(x)_{i}
\end{gathered}
\right\} .
\]
Note that any point in $\left[x_{m}^{n}\right]$ is mapped by $\phi$
into $\left[\phi\left(x_{m}^{n}\right)\right]$.
\end{notation*}

\section{Proof of Theorem \ref{thm:There-exists-a}}

We begin with the following observation:
\begin{lem}
\label{lem:Suppose-that-An}If an ergodic process $(A^{\mathbb{Z}},\mu)$
finitarily factors onto an i.i.d. process $(B^{\mathbb{Z}},p^{\times\mathbb{Z}})$
with entropy $h'$, then the following property holds:

For any $\epsilon>0$, for $\mu$-a.e. $x\in A^{\mathbb{Z}}$, there
exists $L(x)\in\mathbb{N}$ s.t. for any $\ell\geq L(x)$ and any
$r\in\mathbb{N}$, one has
\begin{equation}
\mu\left(\bigcap_{k=0}^{r-1}S^{k\ell}\left[x_{0}^{\ell-1}\right]\right)\leq2^{-(1-\epsilon)h'r\ell}\label{eq:extention inequality}
\end{equation}
\end{lem}

\textbf{Proof}: Assume $\phi:A^{\mathbb{Z}}\rightarrow B^{\mathbb{Z}}$
is a finitary factor map, with $h(p^{\times\mathbb{Z}})=h'\leq h(\mu)$.
Fix $\epsilon>0$. We assume $B$ is finite (otherwise, with arbitrary
small loss of entropy, one can replace $(B^{\mathbb{Z}},p^{\times\mathbb{Z}})$
by a finitary factor of it with finite alphabet). Let $\bar{\epsilon}=\frac{\epsilon h'}{-\log(\min_{b\in B}p_{b})}$.
$\phi$ being finitary, guarantees an $M$ s.t. the set 
\[
X_{M}=\left\{ x\in A^{\mathbb{Z}}:\,R_{0}(x)\leq M\right\} 
\]
 of points whose zero coordinate is encoded at time $M$ is of measure
$\mu(X_{M})>(1-0.1\bar{\epsilon})$. Let $X'\subset A^{\mathbb{Z}}$
be the set of points $x$ that satisfy
\[
\lim_{N\rightarrow\infty}\frac{1}{N}\sum_{n=0}^{N-1}\mathbf{1}\{S^{n}x\in X_{M}\}=\mu(X_{M})
\]
and for each $b\in B$,
\[
\lim_{N\rightarrow\infty}\frac{1}{N}\sum_{n=0}^{N-1}\mathbf{1}\{\phi(x)_{n}=b\}=p_{b}
\]
By the ergodic theorem, this set is of measure 1. For any $x\in X'$,
taking $L(x)$ large enough guarantees that for any $\ell\geq L(x)$,

1. Reading $x_{0}^{\ell-1}$ determines at least $1-0.2\bar{\epsilon}$
of the sample $\phi(x)_{M}^{\ell-M}$.

2. $\frac{2M}{\ell}<0.1\epsilon$ .

3. $\prod_{k=M}^{\ell-M}p_{\phi(x)_{k}}\leq2^{-(1-0.1\epsilon)(\ell-2M)h'}$.

Thus, for any $x\in X'$, $r\in\mathbb{N}$ and $\ell\geq L(x)$,
\begin{align*}
p^{\times\mathbb{Z}}\left(\left[\phi\left(x_{0}^{\ell-1}\right)\right]\right) & \leq2^{-(1-0.1\epsilon)(\ell-2M)h'}\left(\frac{1}{\min_{b\in B}p_{b}}\right)^{0.2\bar{\epsilon}\ell}\\
 & \leq2^{-(1-0.1\epsilon)(1-0.1\epsilon)\ell h'+0.2\epsilon\ell h'}\\
 & \leq2^{-(1-\epsilon)\ell h'}
\end{align*}
and since any point in $\bigcap_{k=0}^{r-1}S^{k\ell}\left[x_{0}^{\ell-1}\right]$
is mapped by $\phi$ into $\bigcap_{k=0}^{r-1}S^{k\ell}\left[\phi\left(x_{0}^{\ell-1}\right)\right]$,
we have
\begin{align*}
\mu\left(\bigcap_{k=0}^{r-1}S^{k\ell}\left[x_{0}^{\ell_{i}-1}\right]\right) & \leq p^{\times\mathbb{Z}}\left(\bigcap_{k=0}^{r-1}S^{k\ell}\left[\phi\left(x_{0}^{\ell-1}\right)\right]\right)\\
 & =p^{\times\mathbb{Z}}\left(\left[\phi\left((x_{0}^{\ell-1}\right)\right]\right)^{r}\\
 & \leq2^{-(1-\epsilon)h'r\ell}
\end{align*}
and the conclusion follows.$\hfill\square$

To prove Theorem \ref{thm:There-exists-a}, we will construct a process
$(A^{\mathbb{Z}},\mu)$ that is a finitary factor of an i.i.d. process,
which for $h'=\log2<h(\mu)$ doesn't satisfy (\ref{eq:extention inequality}).

\textbf{Proof of Theorem \ref{thm:There-exists-a}}: Take $A=\{1,...,10\}$,
and $\mu=(\frac{1}{10},...,\frac{1}{10})^{\mathbb{Z}}$. Let $B=A\times A$,
and enumerate all the words $v\in B^{*}=\cup_{n=1}^{\infty}B^{n}$
as $(v_{L})_{L=10}^{\infty}$, in such a way that for any $L$, 
\begin{equation}
L\geq10|v_{L}|\label{eq:fin.Ber.Copied-word-size}
\end{equation}

We now define a map $\phi:A^{\mathbb{Z}}\rightarrow B^{\mathbb{Z}}$
as follows: For $x\in A^{\mathbb{Z}}$, write $x$ as a concatenation
of Blocks
\[
x=...W_{-1}W_{0}W_{1}...
\]
where each block consists of a ``1'' symbol at the first place,
followed by symbols different from ``1''. Each Block $W_{n}=a_{1}...a_{L}$
is replaced by a $B$-block $V_{n}=b_{1}...b_{L}$ of the same length
(which we denote by $L=|W_{n}|=|V_{n}|$), defined by the following
rule: If $L<10$, then for all $1\leq i\leq L$, let 
\[
b_{i}=(a_{i},a_{i})
\]
 Otherwise, if $L\geq10$, the first $\lfloor L/2\rfloor$ symbols
of $V_{n}$ are filled with consecutive copies of the word $v_{L}$
(the last copy might be cut before its end), and for any $\lfloor L/2\rfloor<i\leq L$,
we let
\[
b_{i}=(a_{i},a_{i-\lfloor L/2\rfloor}).
\]
Finally, define $\phi(x)=y:=...V_{-1}V_{0}V_{1}...$, and $\nu:=\phi_{*}\mu$. 

Evidently, $\phi:(A^{\mathbb{Z}},\mu)\rightarrow(B^{\mathbb{Z}},\nu)$
is a finitary factor map. To estimate its entropy, notice that if
$\psi:A^{\mathbb{Z}}\rightarrow\{0,1\}^{\mathbb{Z}}$ is the projection
on the ``1''-symbol process, defined by 
\[
\psi(x)=\left(\mathbf{1}_{x_{n}=1}\right)_{n\in\mathbb{Z}}
\]
then $\phi\varotimes\psi:(A^{\mathbb{Z}},\mu)\rightarrow(B^{\mathbb{Z}}\times\{0,1\}^{\mathbb{Z}},\left(\phi\varotimes\psi\right)_{*}\mu)$
is an isomorphism, and the marginal measures of $\left(\phi\varotimes\psi\right)_{*}\mu$
are $\phi_{*}\mu=\nu$ and $\psi_{*}\mu$. Thus 
\[
\log10=h\left(\left(\phi\varotimes\psi\right)_{*}\mu\right)\leq h\left(\nu\right)+h\left(\psi_{*}\mu\right)=h\left(\nu\right)+H\left(1/10\right)
\]
where $H\left(1/10\right)=\frac{1}{10}\log10+\frac{9}{10}\log\frac{10}{9}$,
and so,
\[
h(v)\geq\frac{9}{10}\log9>\log2
\]

We now show that $\left(B^{\mathbb{Z}},\nu\right)$ cannot be a finitary
extension of any i.i.d. process of entropy greater or equal $\log2$:
For any $x\in B^{\mathbb{Z}}$ and any $\ell\in\mathbb{N}$, for some
$L\geq10\ell$, $v_{L}=x_{0}^{\ell-1}$. By the construction, we have
that
\[
\nu\left(\bigcap_{k=0}^{\lfloor\frac{L}{2\ell}\rfloor-1}S^{k\ell}[x_{0}^{\ell-1}]\right)\geq\mu\left(\psi(x)_{0}^{L+1}=1\underbrace{0\cdots0}_{{\scriptscriptstyle L-1\,\text{times}}}1\right)=\left(\frac{1}{10}\right)^{2}\left(\frac{9}{10}\right)^{L-1}
\]
Taking $\ell$ to be large, and thus $L$ being large too, 
\begin{align*}
\nu\left(\bigcap_{k=0}^{\lfloor\frac{L}{2\ell}\rfloor}S^{k\ell}[x_{0}^{\ell-1}]\right) & \geq\left(\frac{1}{10}\right)^{2}\left(\frac{9}{10}\right)^{L-1}\\
 & \geq2^{-L/3}\\
 & \geq2^{-\frac{5}{6}\lfloor\frac{L}{2\ell}\rfloor\ell}
\end{align*}
but picking $\epsilon<\frac{1}{6}$ in Lemma \ref{lem:Suppose-that-An}
implies (for large enough $\ell$) 
\[
\nu\left(\bigcap_{k=0}^{\lfloor\frac{L}{2\ell}\rfloor-1}S^{k\ell}[x_{0}^{\ell-1}]\right)<2^{-\frac{5}{6}\lfloor\frac{L}{2\ell}\rfloor\ell}
\]
which is a contradiction.$\hfill\square$

We should remark that the process constructed in the proof of Theorem
\textbf{\ref{thm:There-exists-a}} does finitarily factor onto some
i.i.d. process (of small enough entropy). If one could find an example
of a finitary factor of i.i.d. process that doesn't satisfy the conclusion
of Lemma \ref{lem:Suppose-that-An} for any $h'$, that would imply
the absence of any finitary factor of it which is an i.i.d. process.
However, such an example doesn't exist: 
\begin{prop}
\label{prop:If--is}If $(A^{\mathbb{Z}},\mu)$ is a finitary factor
of an i.i.d. process, then it satisfies the conclusion of Lemma \ref{lem:Suppose-that-An}
for some $h'>0$.
\end{prop}

Thus, it is possible that any finitary factor of i.i.d. process does
finitarily factor onto some i.i.d. process, but we have not settled
this question. 

\textbf{Proof of Proposition }\ref{prop:If--is}: Let $\phi:(B^{\mathbb{Z}},p^{\times\mathbb{Z}})\rightarrow(A^{\mathbb{Z}},\mu)$
be the finitary factor map, fix $a\in A$ with $0<\mu(\{x\in A^{\mathbb{Z}}:\,x_{0}=a\})<1$,
and let $w\in B^{*}$ be a name of length $|w|=2k+1$ such that $\phi_{0}\equiv a$
$p^{\times\mathbb{Z}}$-a.s. on the set $\{y\in B^{\mathbb{Z}}:\,y_{-k}^{k}=w\}$.
Define 
\[
h':=-\log\left(1-p^{\times\mathbb{Z}}\left(y_{-k}^{k}=w\right)\right)\frac{1-\mu(x_{0}=a)}{2|w|}
\]
 Being a factor of Bernoulli system, $\mu$ is totally ergodic. Thus,
for a.e. $x$, taking $L(x)$ large enough guarantees that for any
$\ell\geq L(x)$, the set 
\[
D(x,\ell):=\left\{ 0\leq j<\lfloor\frac{\ell}{|w|}\rfloor:\,(x)_{|w|j}\neq a\right\} 
\]
satisfies
\[
\#D(x,\ell)\geq\frac{1}{2}\cdot\frac{\ell}{|w|}(1-\mu(x_{0}=a))
\]
 I addition, we have 
\[
\phi^{-1}\left(\left[x_{0}^{\ell-1}\right]\right)\subset\bigcap_{j\in D(x,\ell)}S^{|w|j}\left\{ y\in B^{\mathbb{Z}}:\,y_{-k}^{k}\neq w\right\} 
\]
and thus for any $r\in\mathbb{N}$,
\begin{align*}
\mu\left(\bigcap_{k=0}^{r-1}S^{k\ell}\left[x_{0}^{\ell-1}\right]\right) & =p^{\times\mathbb{Z}}\left(\phi^{-1}\left(\bigcap_{k=0}^{r-1}S^{k\ell}\left[x_{0}^{\ell-1}\right]\right)\right)\\
 & =p^{\times\mathbb{Z}}\left(\bigcap_{k=0}^{r-1}S^{k\ell}\phi^{-1}\left(\left[x_{0}^{\ell-1}\right]\right)\right)\\
 & \leq p^{\times\mathbb{Z}}\left(\bigcap_{k=0}^{r-1}\bigcap_{j\in D(x,\ell)}S^{k\ell+|w|j}\left\{ y\in B^{\mathbb{Z}}:\,y_{-k}^{k}\neq w\right\} \right)\\
 & \leq\left(1-p^{\times\mathbb{Z}}\left(y_{-k}^{k}=w\right)\right)^{\frac{1-\mu(x_{0}=a)}{2|w|}r\ell}\\
 & =2^{-h'r\ell}
\end{align*}
$\hfill\square$

\section{Proof of Theorem \ref{thm:Finitary Sinai}}

Let $\mathcal{X}$ be an ergodic system of finite entropy. Since any
such system is isomorphic to a process, we can assume that $\mathcal{X}=(A^{\mathbb{Z}},\mu)$
is a process. In order to prove the theorem, we use a residual version
of Krieger's finite generator theorem, following \cite[Section 7.5]{key-13}.
For the clarity of presentation, we state the theorem not in its full
extent. We now give the definitions and the statement of that theorem.

Let $\mathcal{X}=\left(A^{\mathbb{Z}},\mu\right)$ be an ergodic process
of entropy $h<\infty$, and $J$ be the compact space of all joinings
$\lambda$ of $\mu$ with any shift-invariant measure on $A^{\mathbb{Z}}$,
endowed with the weak{*} topology metric
\[
d(\lambda,\lambda')=\sum_{n=1}^{\infty}2^{-n}\left(\frac{1}{2}\sum_{v\in\left(A\times A\right)^{n}}|\lambda([v])-\lambda'([v])|\right)
\]
 Let $\tilde{J}\subset J$ be the set of all $\lambda\in J$ which
are ergodic, and such that their second marginal has entropy strictly
greater than that of $\mu$. Finally, let $\hat{J}$ be the weak{*}
closure of $\tilde{J}$ in $J$. With the induced metric, $\hat{J}$
is a compact metric space.

For a measurable map $\phi:A^{\mathbb{Z}}\rightarrow A^{\mathbb{Z}}$,
we say that a joining $\lambda\in J$ is supported on the graph of
$\phi$ if $\lambda\left((x,\phi(x)):\,x\in A^{\mathbb{Z}}\right)=1$.
In this case, $\phi_{*}\mu$ is the second marginal of $\lambda$.
\begin{thm}
(Residual Krieger's finite generator theorem) The set $\mathcal{I}\subset\hat{J}$
of joinings supported on a graph of an isomorphism is $G_{\delta}$-dense
in $\hat{J}$.
\end{thm}

We will show that there exists a $G_{\delta}$-dense set of joinings
$\mathcal{O}\subset\hat{J}$, for which the second marginal measure
has no finitary Bernoulli factors. Thus the intersection $\mathcal{I}\cap\mathcal{O}$
is not empty, and the second marginal of any joining in it is an isomorphic
copy of $\left(A^{\mathbb{Z}},\mu\right)$ which has no finitary Bernoulli
factors. This proves the theorem for all processes.

We write a point in the measure space $A^{\mathbb{Z}}\times A^{\mathbb{Z}}=(A\times A)^{\mathbb{Z}}$
as a sequence of tuples: 
\[
x=(x_{1}^{n},x_{2}^{n})_{n\in\mathbb{Z}}
\]
 For a name $w\in A^{*}={\displaystyle \bigcup_{n=1}^{\infty}A^{n}}$
and an integer $M\in\mathbb{N}$, we write
\[
[w^{M}]=\left\{ x\in A^{\mathbb{Z}}:\,x_{0}^{M|w|-1}=\underbrace{ww\cdots w}_{{\scriptscriptstyle M\,\text{times}}}\right\} 
\]
(for $M=1$ we simply write $[w]$). Fix $w\in A^{*}$ and define
$\mathcal{O}(w)\subset\hat{J}$ by
\[
\mathcal{O}(w)=\left\{ \lambda\in\hat{J}:\:\exists M\in\mathbb{N},\,\lambda\left(A^{\mathbb{Z}}\times[w^{\lceil\log(2M)\rceil}]\right)>\frac{1}{M}\right\} .
\]
\begin{prop}
\label{prop:-is-open.}$\mathcal{O}(w)$ is open.
\end{prop}

\begin{proof}
For a fixed $M$, the set $A^{\mathbb{Z}}\times[w^{\lceil\log(2M)\rceil}]$
is clopen in $A^{\mathbb{Z}}\times A^{\mathbb{Z}}$, and consequently,
$\mathbf{1}_{A^{\mathbb{Z}}\times[w^{\lceil\log(2M)\rceil}]}$ is
continuous. This in turn implies that $\mathcal{O}(w)$ is a union
of open sets in the weak{*} topology.
\end{proof}
\begin{prop}
\label{prop:Let--be}$\mathcal{O}(w)$ is dense in $\hat{J}$.
\end{prop}

\begin{proof}
Since $\hat{J}$ is the closure of $\tilde{J}$, it suffices to show
that for any $\lambda\in\tilde{J}$ and any $\epsilon>0$, there exists
$\lambda'\in\tilde{J}$ with $d(\lambda,\lambda')<\epsilon$, which
satisfies for some $M$
\[
\lambda'\left(A^{\mathbb{Z}}\times[w^{\lceil\log(2M)\rceil}]\right)>\frac{1}{M}
\]

Fix $\lambda\in\tilde{J}$ and $\epsilon>0$. Let $\nu$ be the second
marginal measure of $\lambda$. Let $M$ be a large integer depending
on $\epsilon$ and $|w|$ that will be determined later, and take
a measurable subset $B\subset A^{\mathbb{Z}}$ which satisfies:

(i) $\nu(B)=\frac{2}{M}$.

(ii) The sets $B,SB,...,S^{|w|\cdot\lceil\log(2M)\rceil-1}B$ are
pairwise disjoint.

Assuming $M$ is large enough so that $|w|\cdot\lceil\log(2M)\rceil<M/2$,
one can always pick such a $B$: For instance, let $0<\alpha<1$ be
such that $|w|\cdot\lceil\log(2M)\rceil<\alpha M/2$, and take a Rokhlin
tower of height $|w|\cdot\lceil\log(2M)\rceil$ with total measure
greater than $\alpha$. Then its base is of measure greater than $\frac{2}{M}$,
and any subset of its base of measure $\frac{2}{M}$ will work for
that $B$. 

Define $\Phi:A^{\mathbb{Z}}\times A^{\mathbb{Z}}\rightarrow A^{\mathbb{Z}}\times A^{\mathbb{Z}}$
by $\Phi:=Id_{A^{\mathbb{Z}}}\times\left(\phi_{n}\right)_{n\in\mathbb{Z}}$,
where for any $n$, $\phi_{n}:=\phi_{0}\circ S^{n}$, and for any
$a=(a_{n})_{n\in\mathbb{Z}}\in A^{\mathbb{Z}}$,
\[
\phi_{0}(a):={\displaystyle \begin{cases}
w_{j} & a\in S^{m|w|+j}B\hspace{1em}(0\leq j<|w|,0\leq m<\log2M)\\
a_{0} & a\in A^{\mathbb{Z}}\backslash{\displaystyle {\textstyle {\displaystyle \bigcup_{i=0}^{|w|\lceil\log2M\rceil-1}}}}S^{i}B
\end{cases}}
\]

Let $\lambda'=\Phi_{*}\lambda$, and denote by $\nu'$ its second
marginal (clearly, $\mu$ is its first marginal). We claim that $\lambda'\in\mathcal{O}(w)$,
and moreover, if $M$ is picked to be large enough, one has $d(\lambda,\lambda')<\epsilon$:
First notice that $\lambda'$ is ergodic, since it is the pushforward
of an ergodic measure. In addition, for any $y\in A^{\mathbb{Z}}\times B$,
one has that $\Phi(y)\in A^{\mathbb{Z}}\times[w^{\lceil\log(2M)\rceil}]$,
thus
\[
\lambda'(A^{\mathbb{Z}}\times[w^{\lceil\log(2M)\rceil}])\geq\lambda(A^{\mathbb{Z}}\times B)=\frac{2}{M}>\frac{1}{M}
\]
It remains to show that $h(\nu')>h(\mu)$ and that $d(\lambda,\lambda')<\epsilon$.
For the entropy inequality we use the $\bar{d}$-distance between
$\nu$ and $\nu'$, while for the latter inequality we use the $\bar{d}$-distance
between $\lambda$ and $\lambda'$:
\begin{align}
\bar{d}(\nu,\nu'): & =\min\left\{ \int_{A^{\mathbb{Z}}\times A^{\mathbb{Z}}}\mathbf{1}\{x_{0}\neq y_{0}\}d\rho(x,y):\,\rho\text{ a joining of \ensuremath{\nu} and \ensuremath{\nu'}}\right\} \nonumber \\
 & \leq\int_{A^{\mathbb{Z}}}\mathbf{1}\{x_{0}\neq\phi_{0}(x)\}d\nu(x)\label{eq:dbar}\\
 & \leq\nu\left(\bigcup_{i=0}^{|w|\lceil\log2M\rceil-1}S^{i}B\right)\nonumber \\
 & =\frac{2|w|\cdot\lceil\log(2M)\rceil}{M}\nonumber 
\end{align}
 which tends to zero as $M\longrightarrow\infty$. Since process entropy
is continuous with respect to $\bar{d}$, one has $h(\nu')\overset{{\scriptscriptstyle M\rightarrow\infty}}{\longrightarrow}h(\nu)$,
and in particular, for large enough $M$, $h(\nu')>h(\mu)$. Thus
$\lambda'\in\mathcal{O}(w)$. Similarly to (\ref{eq:dbar}), we have
\begin{align*}
\bar{d}(\lambda,\lambda'): & =\min\left\{ \int_{\left(A^{\mathbb{Z}}\times A^{\mathbb{Z}}\right)^{2}}\mathbf{1}\{x_{0}\neq y_{0}\}d\rho(x,y):\,\rho\text{ a joining of \ensuremath{\lambda} and \ensuremath{\lambda'}}\right\} \\
 & \leq\int_{A^{\mathbb{Z}}\times A^{\mathbb{Z}}}\mathbf{1}\{x_{0}\neq\Phi_{0}(x)\}d\lambda(x)\\
 & \leq\lambda\left(\bigcup_{i=0}^{|w|\lceil\log2M\rceil-1}\left(A^{\mathbb{Z}}\times S^{i}B\right)\right)\\
 & =\frac{2|w|\cdot\lceil\log(2M)\rceil}{M}
\end{align*}
 which tends to zero as $M\longrightarrow\infty$. Since $d$ is continuous
with respect to $\bar{d}$, taking large enough $M$ gives $d(\lambda,\lambda')<\epsilon$.
\end{proof}
By propositions \ref{prop:-is-open.} and \ref{prop:Let--be}, the
set 
\[
\mathcal{O}=\bigcap_{w\in A^{*}}\mathcal{O}(w)
\]
 is $G_{\delta}$-dense.
\begin{prop}
For any $\lambda\in\mathcal{O}$, its second marginal measure has
no finitary factors which are i.i.d.
\end{prop}

\begin{proof}
Let $\lambda\in\mathcal{O}$. Denote by $\nu$ its second marginal
measure, and suppose that there is some finitary factor map $\phi$
from $(A^{\mathbb{Z}},\nu)$ to some i.i.d. process $\mathbf{\mathcal{B}}=(\{1,...,s\}^{\mathbb{Z}},p^{\times\mathbb{Z}},\sigma)$,
where $p=(p_{1},...,p_{s})$ is a probability measure on the state
space $\mathbf{B}=\{1,...,s\}$. Evidently, one of the symbols $b\in B$
satisfies $0<p_{b}\leq\frac{1}{2}$. Since $\phi$ is finitary, there
is some finite word $w\in A^{*}$ of odd length $2k+1$, so that $\left\{ x:x_{-k}^{k}=w\right\} \subset\phi_{0}^{-1}(a)$.
Thus, for any $x\in A^{\mathbb{Z}}$ and $n\in\mathbb{Z}$, if $x_{n-k}^{n+k}=w$,
then $\phi_{n}(x)=a$. Thus for all $M$,
\[
\nu([w^{\lceil\log(2M)\rceil}])\leq p^{\times\mathbb{Z}}(a)^{\lceil\log(2M)\rceil}\leq\frac{1}{2M}
\]
On the other hand, for some $M$ (by definition of $\mathcal{O}(w)$),
\[
\nu([w^{\lceil\log(2M)\rceil}])=\lambda\left(A^{\mathbb{Z}}\times[w^{\lceil\log(2M)\rceil}]\right)>\frac{1}{M}
\]
which is a contradiction.
\end{proof}

\section*{Open Problems}
\begin{enumerate}
\item Does there exists a finitary factor of an i.i.d. process non of whose
finitary factors is an i.i.d. process (regardless of the factor's
entropy)?
\item (Suggested by Tim Austin) Fix an alphabet $A$ and an i.i.d. process
$\mathcal{B=}(B^{\mathbb{Z}},p^{\times\mathbb{Z}})$. Let $\mathfrak{M}$
be the space of all shift invariant measures on $A^{\mathbb{Z}}$,
endowed with the weak{*} topology. 
\begin{enumerate}
\item Is $\left\{ \mu\in\mathfrak{M}:\,(A^{\mathbb{Z}},\mu)\text{ is finitarily isomorphic to \ensuremath{\mathcal{B}}}\right\} $
a Borel subset in $\mathfrak{M}$?
\item Is $\left\{ \mu\in\mathfrak{M}:\,(A^{\mathbb{Z}},\mu)\text{ finitarily factors onto \ensuremath{\mathcal{B}}}\right\} $
a Borel subset in $\mathfrak{M}$?
\end{enumerate}
\end{enumerate}
\begin{acknowledgement*}
This paper is part of the author\textquoteright s Ph.D. thesis, conducted
under the guidance of Professor Michael Hochman, whom I would like
to thank for all his support and advice. Also thanks to Tim Austin
and Lewis Bowen for valuable discussions.
\end{acknowledgement*}

\lyxaddress{uriel.gabor@gmail.com}

\lyxaddress{Einstein Institute of Mathematics}

\lyxaddress{The Hebrew University of Jerusalem}

\lyxaddress{Edmond J. Safra Campus, Jerusalem, 91904, Israel}
\end{document}